\documentclass[12pt]{article}
\usepackage[T1]{fontenc}
\usepackage[mathscr]{euscript}
\usepackage{amsmath}
\usepackage{amsthm}
\usepackage{stackrel, amssymb}
\usepackage{physics}
\usepackage{tensor}
\usepackage{mathtools}
\usepackage{slashed}
\usepackage{quiver}
\usepackage{hyperref}
\hypersetup{ 
	colorlinks=true,
	linkcolor=blue,
	urlcolor=blue,
	citecolor=blue
}
\usepackage{geometry}
\usepackage{sectsty}
\subsectionfont{\normalfont\itshape}
\subsubsectionfont{\normalfont\itshape}
\usepackage[toc, page]{appendix}
\usepackage{authblk}
\usepackage{xcolor}
\usepackage[style=alphabetic, maxbibnames=99, maxalphanames=99]{biblatex}
\numberwithin{equation}{section}

\newtheorem{prop}{Proposition}[section]

\newtheorem{theorem}[prop]{Theorem}

\newtheorem{lemma}[prop]{Lemma}

\newtheorem{cor}[prop]{Corollary}

\title{Quasimaps to Nakajima varieties as critical loci}
\author{Spencer Tamagni}
\affil{\textit{Leinweber Institute for Theoretical Physics, University of California, Berkeley}}
\date{\today}
\begin{document}

\maketitle

\begin{abstract}
$X$ is a Nakajima quiver variety. In this note we report the observation that the space parameterizing quasimaps from $\mathbb{P}^1$ to $X$, sending infinity to a given point, may be presented globally as the critical locus of an explicit function.
\end{abstract}

\setcounter{tocdepth}{2}
\tableofcontents 

\section{Introduction}
In this short note we will give an ADHM-style \cite{adhm} construction of the moduli space of quasimaps from $\mathbb{P}^1$ to a Nakajima quiver variety $X$. Our construction is reasonably canonical and presents this space as the critical locus of an explicit function. A consequence of this is that the space of quasimaps has a natural derived enhancement with a $(-1)$-shifted symplectic structure in the sense of \cite{ptvv}, \textit{which can be described explicitly}. More details on the derived structure of quasimap spaces will appear in a companion paper \cite{BT}. 

In the instructive special case where $X = \text{Hilb}_n(\mathbb{C}^2)$, the construction may be described as follows. Consider the following quiver \eqref{bigadhmquiver}.

\begin{equation} \label{bigadhmquiver}
\begin{tikzcd}[row sep = huge]
    \underline{W} \arrow[r, "\iota", shift left=1.5] & \arrow[l, "j"]  \underline{V} \arrow[out=0, in=45, loop, swap, "\xi_2"] 
    \arrow[out=90, in=135, loop, swap, "\xi_3"] 
    \arrow[d, swap, "I"] \\
    & U
    \arrow[out=157.5, in=202.5, loop, swap, "B_1"]
    \arrow[out=247.5, in=292.5, loop, swap, "B_2"]
    \arrow[out=337.5, in=22.5, loop, swap, "B_3"] 
    \arrow[u, bend right, swap, "J_2"]
    \arrow[u, bend right, shift right=3ex, swap, "J_3"]
    \arrow[ul, "\alpha"]
\end{tikzcd}
\end{equation}

We view the underlined vector spaces $W$, $V$ (of respective dimensions $1$ and $n$) as fixed framing spaces and likewise view the ADHM data $(\xi_2, \xi_3, \iota, \jmath)$ as describing a given point $p$ in $\text{Hilb}_n(\mathbb{C}^2)$. See Section \ref{nak} below for detailed notation and conventions on Nakajima quiver varieties. Write the potential function, which depends on the point $p$ (in particular, we view $\xi_2, \xi_3, \iota, \jmath$ as fixed and do not differentiate with respect to them when taking the critical locus below):
\begin{equation} \label{potential}
    \mathscr{W}_p = \tr(B_1 \comm{B_2}{B_3}) + \tr(J_2(B_2 I - I \xi_2)) + \tr(J_3 (B_3 I - I \xi_3)) + \tr(I \iota \alpha). 
\end{equation}
Notice this function is invariant under the action of the group $GL(U) \ltimes \text{Hom}(U, V)$ defined by \eqref{eq:groupact} below. Write $\overline{\mathscr{W}}_p$ for the induced function on the quotient by $GL(U) \ltimes \text{Hom}(U, V)$ of the locus of $\textit{stable}$ representations satisfying $\mathbb{C}[B_1]I(V) = U$. Write for the critical locus
\begin{equation}
    M(p, d) := \text{Crit}(\overline{\mathscr{W}}_p).
\end{equation}
Let $\mathsf{QM}^p_d$ denote the moduli space of stable quasimaps from $\mathbb{P}^1$ to $\text{Hilb}_n(\mathbb{C}^2)$ of degree $d = \dim(U)$ and such that $f(\infty) = p$, see Section \ref{quasidef} below or \cite{qmtogit}. Then we have 
\begin{theorem} \label{mainiso}
There is an isomorphism of algebraic stacks \begin{tikzcd} \mathsf{QM}^p_d \arrow[r, "\sim"] & M(p, d). \end{tikzcd}
\end{theorem}
It is no more difficult to establish this result for Nakajima quiver varieties in general, which is formulated as Theorem \ref{thm:quasimapcrit} below. Note in the special case $X = \text{Hilb}_n(\mathbb{C}^2)$, we obtain an explicit critical locus description of a certain Pandharipande-Thomas moduli space of stable pairs \cite{Pandharipande_2009}, see \cite{okounkovpcmi} Exercise 4.3.22. 

The only somewhat nonstandard feature of our presentation (at least, in the context of quiver representations) is the quotient by the non-reductive group $GL(U) \ltimes \text{Hom}(U, V)$. It is explained in Section \ref{sect:splittings} below that there is an equivalent description of the quiver moduli space involving a quotient by $GL(U)$ only, though this presentation depends non-canonically on the choice of point $p$ and some auxiliary additional data. In other words, it is possible to gauge fix the non-reductive part of the quotient, though as usual the gauge fixing condition involves some additional choices. 

Theorem \ref{thm:quasimapcrit} is proven in the main body by providing an explicit construction of the maps in both directions and checking they are inverse to each other. The proof is very similar to the analogous construction identifying handsaw quiver varieties with spaces of quasimaps to flag manifolds \cite{nakajima2011hansaw}, though some interesting new features arise (namely, the need to explicitly make use of a potential function). The implications of this construction for geometric representation theory will be explored in forthcoming joint work with T. Botta \cite{BT}. 

\subsubsection{Acknowledgements}
I am grateful to Mina Aganagic, Andrei Okounkov, Vivek Shende, and Peng Zhou for useful discussions and encouragement. I am especially grateful to Tommaso Botta for collaboration on a related project and for initially posing the question which this note answers in the affirmative. 

\section{Preliminaries}
All algebraic geometry in this paper is over $\mathbb{C}$. 
\subsection{Nakajima quiver varieties} \label{nak}
Let $Q$ be a quiver, $Q_0$ its vertex set and $Q_1$ its edge set, each of which we assume finite. Fix an orientation and denote the associated source and target maps on edges as $s, t: Q_1 \to Q_0$. Fix dimension vectors $\vb{v}, \vb{w}: Q_0 \to \mathbb{Z}_{\geq 0}$, with components $v_i, w_i$, $i \in Q_0$. $\vb{w}$ is called the framing vector. A framed quiver representation is an assignment of vector spaces $V_i$ of dimension $v_i$ to each $i \in Q_0$ and maps between them to each edge $e \in Q_1$, as well as maps $W_i \to V_i$ for each $i \in Q_0$, where $\dim W_i = w_i$. The vector space of representations at fixed dimension vectors is denoted 
\begin{equation}
    \text{Rep}(\vb{v}, \vb{w}) := \bigoplus_{e \in Q_1} \text{Hom}(V_{s(e)}, V_{t(e)}) \oplus \bigoplus_{i \in Q_0} \text{Hom}(W_i, V_i). 
\end{equation}
The group $G_{\vb{v}} =\prod_{i \in Q_0} GL(V_i)$ acts naturally on $\text{Rep}(\vb{v}, \vb{w})$; the induced action on the cotangent bundle $T^* \text{Rep}(\vb{v}, \vb{w})$ has moment map $\mu: T^* \text{Rep}(\vb{v}, \vb{w}) \to \mathfrak{g}^*_{\vb{v}}$. If we denote a point in $T^* \text{Rep}(\vb{v}, \vb{w})$ as $(\xi_{2, e}, \xi_{3, e}, \iota_i, \jmath_i)_{e \in Q_1, i \in Q_0}$, the moment map may be written compactly as 
\begin{equation}
    \mu = \comm{\xi_2}{\xi_3} + \iota \jmath. 
\end{equation}
The Nakajima variety associated to $Q, \vb{v}, \vb{w}$ is defined by 
\begin{equation}
    X(\vb{v}, \vb{w}) := \mu^{-1}(0)^{st}/G_{\vb{v}}. 
\end{equation}
$\mu^{-1}(0)^{st} \subset \mu^{-1}(0)$ is the open subset of quiver data satisfying the \textit{stability condition}: any collection $S$ of subspaces $(S_i \subseteq V_i)_{i \in Q_0}$ which are $\xi_2, \xi_3$ invariant and such that $S_i \subseteq \ker \jmath_i$ must have $S_i = 0$, $i \in Q_0$. This is also the GIT or $\theta$-stability condition for $\theta_i = - 1$, $i \in Q_0$. $G_{\vb{v}}$ acts freely on the stable locus and the quotient is a smooth quasiprojective variety \cite{nakajimaALE}, \cite{nakajimaKM}.

It will be convenient for the analysis below to fix a point $p \in X(\vb{v}, \vb{w})$; denote the quiver data assigned to $p$ as $(\xi_{2, e}, \xi_{3, e}, \iota_i, \jmath_i)_{e \in Q_1, i \in Q_0}$. From now on we will drop $\vb{v}, \vb{w}$ from the notation and denote the Nakajima variety by $X$.

\subsection{Quasimaps to Nakajima varieties} \label{quasidef}
The next preliminary necessary for this paper is the moduli space of quasimaps to $X$. See \cite{okounkovpcmi} for an accessible introduction to quasimaps to Nakajima varieties and \cite{qmtogit} for their detailed construction. 

The moduli space of quasimaps from $\mathbb{P}^1$ to $X$ is by definition the open substack 
\begin{equation}
    \mathsf{QM}(X) \subset \text{Map}(\mathbb{P}^1, [\mu^{-1}(0)/G_{\vb{v}}])
\end{equation}
of maps to the quotient stack $[\mu^{-1}(0)/G_{\vb{v}}]$ which send the generic point of $\mathbb{P}^1$ to the stable locus $X \subset [\mu^{-1}(0)/G_{\vb{v}}]$. We have suppressed $\mathbb{P}^1$ from the notation because this note will only consider quasimaps from a fixed source curve $\mathbb{P}^1$. 

We choose a distinguished point $\infty \in \mathbb{P}^1$. There is an evaluation map 
\begin{equation}
    \text{ev}_\infty: \mathsf{QM}(X) \to [\mu^{-1}(0)/G_{\vb{v}}]
\end{equation}
and we will consider the moduli space 
\begin{equation}
    \mathsf{QM}^p(X) := \text{ev}^{-1}_\infty(p)
\end{equation}
for our chosen point $p \in X \subset [\mu^{-1}(0)/G_{\vb{v}}]$. 

\subsubsection{Quasimap data} \label{sect:qmapdata}
Very explicitly, the actual data parameterized by $\mathsf{QM}^p(X)$ is the following. $B$ is an arbitrary base scheme; a family of quasimaps over $B$ consists of the following data. For each $i \in Q_0$ we have a vector bundle $\mathscr{V}_i$ on $B \times \mathbb{P}^1$ of rank $v_i$, together with morphisms 
\begin{equation}
\begin{split}
    \mathscr{B}_{2, e}  & \in H^0(B \times \mathbb{P}^1, \mathscr{H}om(\mathscr{V}_{s(e)}, \mathscr{V}_{t(e)})) \\ 
    \mathscr{B}_{3, e} & \in H^0(B \times \mathbb{P}^1, \mathscr{H}om(\mathscr{V}_{t(e)}, \mathscr{V}_{s(e)})) \\ 
    \mathscr{I}_i & \in H^0(B \times \mathbb{P}^1, \mathscr{H}om(\mathscr{W}_i, \mathscr{V}_i)) \\ 
    \mathscr{J}_i & \in H^0(B \times \mathbb{P}^1, \mathscr{H}om(\mathscr{V}_i, \mathscr{W}_i))
\end{split}
\end{equation}
where $\mathscr{W}_i := W_i \otimes \mathscr{O}_{B \times \mathbb{P}^1}$ denote fixed trivial bundles associated to the framing spaces $W_i$. These sections must satisfy the moment map conditions
\begin{equation}
    \comm{\mathscr{B}_2}{\mathscr{B}_3} + \mathscr{I} \mathscr{J} = 0 \in \bigoplus_{i \in Q_0} H^0(B \times \mathbb{P}^1, \mathscr{H}om(\mathscr{V}_i, \mathscr{V}_i)). 
\end{equation}
Moreover in the fiber over $\infty \in \mathbb{P}^1$ there is an identification 
\begin{equation}
    \mathscr{V}_i \eval_{B \times \infty} \simeq V_i \otimes \mathscr{O}_B
\end{equation}
with respect to which $\mathscr{B}_2(\infty) = \xi_2$, $\mathscr{B}_3(\infty) = \xi_3$, $\mathscr{I}(\infty) = \iota$, $\mathscr{J}(\infty) = \jmath$ corresponding to the point $p = (\xi_2, \xi_3, \iota, \jmath)$. Note that, as $p$ is in the stable locus, it is automatic that (for each point in $B$) the quasimap sends a Zariski open neighborhood of $\infty \in \mathbb{P}^1$ to the stable locus, so stability at the generic point is satisfied. 

Under morphisms $B' \to B$, families of quasimaps pull back by pulling back bundles and sections. An isomorphism of quasimaps $(\mathscr{V}, \mathscr{B}_2, \mathscr{B}_3, \mathscr{I}, \mathscr{J})$ and $(\mathscr{V}', \mathscr{B}'_2, \mathscr{B}'_3, \mathscr{I}', \mathscr{J}')$ over $B$ is an isomorphism
\begin{equation}
\begin{tikzcd}
    \phi: \mathscr{V} \arrow[r, "\sim"] & \mathscr{V}'
\end{tikzcd}
\end{equation}
of the underlying bundles on $B \times \mathbb{P}^1$, i.e. a tuple of isomorphisms $(\phi_i: \mathscr{V}_i \longrightarrow \mathscr{V}_i')_{i \in Q_0}$, which in addition is compatible with the sections. Compatibility with the sections means that the following squares in coherent sheaves on $B \times \mathbb{P}^1$ commute: 
\begin{equation} \label{quasimapisodiagrams}
\begin{tikzcd}
    \mathscr{V} \arrow[r, "\phi"] \arrow[d, "\mathscr{B}_\alpha"] & \mathscr{V}' \arrow[d, "\mathscr{B}'_\alpha"]\\
    \mathscr{V} \arrow[r, "\phi"] & \mathscr{V}'
\end{tikzcd} 
\, \, \, \text{and} \, \, \, 
\begin{tikzcd}[ampersand replacement=\&]
    \mathscr{V} \arrow[r, "\phi"] \& \mathscr{V}' \\
    \mathscr{W} \arrow[u, "\mathscr{I}"] \arrow{ur}[swap]{\mathscr{I}'}
\end{tikzcd}
\, \, \, \text{and} \, \, \, 
\begin{tikzcd}[ampersand replacement=\&]
   \mathscr{V} \arrow[r, "\phi"] \arrow[d, "\mathscr{J}"] \& \mathscr{V}' \arrow[ld, "\mathscr{J}'"] \\ 
   \mathscr{W}
\end{tikzcd}
\end{equation}
With respect to the identifications $\mathscr{V}_i \eval_{B \times \infty} \simeq \mathscr{V}'_i \eval_{B \times \infty} \simeq V_i \otimes \mathscr{O}_B$, we must have $\phi_i \eval_{B \times \infty} = 1$ since $G_{\vb{v}}$ acts freely on the stable locus. For this reason, in the rest of the paper we will regard the vector space $V_i$ as fixed and use the identification $\mathscr{V}_i \eval_{B \times \infty} \simeq V_i \otimes \mathscr{O}_B$ without comment. By a similar logic, quasimaps we consider have no nontrivial automorphisms. 

Sometimes it will be convenient to use the shorthand $f = (\mathscr{V}, \mathscr{B}_2, \mathscr{B}_3, \mathscr{I}, \mathscr{J})$ for the quasimap data. We may informally write 
\begin{equation}
    \mathsf{QM}^p(X) = \{ (\mathscr{V}, \mathscr{B}_2, \mathscr{B}_3, \mathscr{I}, \mathscr{J}) | \comm{\mathscr{B}_2}{\mathscr{B}_3} + \mathscr{I} \mathscr{J} = 0, \, \, \, f(\infty) = p\}/\text{isomorphisms}
\end{equation}
as shorthand for the corresponding moduli stack. The bundles $\mathscr{V}_i$ entering the quasimap data induce universal bundles on $\mathsf{QM}^p(X) \times \mathbb{P}^1$, and likewise the sections induce universal morphisms between these bundles. By abuse of notation, we will use the same symbols for the quasimap data at an arbitrary geometric point of the moduli space and the corresponding universal data; this will hopefully not cause any confusion because, unless explicitly stated otherwise, we will never have the need to take the fiber at any point. 

\subsubsection{Fixing the degree}
There is a locally constant function $\deg : \mathsf{QM}^p(X) \to \mathbb{Z}^{Q_0}$ defined by 
\begin{equation}
    (\mathscr{V}, \mathscr{B}_2, \mathscr{B}_3, \mathscr{I}, \mathscr{J}) \mapsto (-\deg \mathscr{V}_i)_{i \in Q_0}. 
\end{equation}
For a fixed tuple $d = (d_i)_{i \in Q_0}$, let 
\begin{equation}
    \mathsf{QM}^p_d(X) := \deg^{-1}(d).
\end{equation}
It follows from Lemma \ref{lemma:canonicalsplit} below that $\mathsf{QM}^p_d(X)$ is nonempty only if $d_i \geq 0$.

\subsection{The quiver moduli space} \label{sect:quivermoduli}
The final player necessary for this paper is the following moduli space of quiver representations. It depends on a choice of point $p = (\xi_2, \xi_3, \iota, \jmath) \in X$ as above, which one should think of as a ``frozen'' quiver representation, upon which we build a new one by a certain enlargement procedure. 

Consider the quiver $Q$ underlying $X$, with dimension vectors $\vb{v}$ and $\vb{w}$ as usual. Fix a new dimension vector $d: Q_0 \to \mathbb{Z}_{\geq 0}$ and add a vector space $U_i$ to each node of the quiver, of dimension $\dim U_i = d_i$. Define the vector space 
\begin{equation}
\begin{split}
    \mathscr{R}(\vb{v}, \vb{w}; d) := & \bigoplus_{i \in Q_0} \text{Hom}(U_i, U_i) \oplus \text{Hom}(V_i, U_i)\oplus \text{Hom}(U_i, W_i) \\
    & \bigoplus_{e \in Q_1} \text{Hom}(U_{s(e)}, U_{t(e)}) \oplus \text{Hom}(U_{s(e)}, V_{t(e)}) \oplus \text{Hom}(U_{t(e)}, U_{s(e)}) \oplus \text{Hom}(U_{t(e)}, V_{s(e)}).  
\end{split}
\end{equation}
An element of $\mathscr{R}(\vb{v}, \vb{w}; d)$ is denoted by $$(B_{1, i}, I_i, \alpha_i, B_{2, e}, J_{3, e}, B_{3, e}, J_{2, e})_{i \in Q_0, e \in Q_1}$$ or $(B_1, I, \alpha, B_2, J_3, B_3, J_2)$ for short. Define the function $\mathscr{W}_p : \mathscr{R}(\vb{v}, \vb{w}; d) \to \mathbb{C}$ by 
\begin{equation} \label{eq:definepotential}
\begin{split}
    \mathscr{W}_p & := \tr B_1 \comm{B_2}{B_3} + \tr J_2(B_2 I - I \xi_2) + \tr J_3(B_3 I - I \xi_3) + \tr I \iota \alpha \\
    & := \sum_{\substack{i \in Q_0 \\ e \in t^{-1}(i)}} \tr B_{1, i} B_{2, e} B_{3, e} - \sum_{\substack{i \in Q_0 \\ e \in s^{-1}(i)}} \tr B_{1, i} B_{3, e} B_{2, e} + \sum_{e \in Q_1} \tr J_{2, e} (B_{2, e} I_{s(e)} - I_{t(e)} \xi_{2, e}) \\
    & \qquad \qquad \qquad \qquad \qquad \qquad + \sum_{e \in Q_1} \tr J_{3, e}( B_{3, e} I_{t(e)} - I_{s(e)} \xi_{3, e}) + \sum_{i \in Q_0} \tr I_i \iota_i \alpha_i.
\end{split}
\end{equation}
On $\mathscr{R}(\vb{v}, \vb{w}; d)$ acts the group 
\begin{equation}
    GL(U) \ltimes \text{Hom}(U, V) := \prod_{i \in Q_0} GL(U_i) \ltimes \text{Hom}(U_i, V_i) \hookrightarrow \prod_{i \in Q_0} GL(U_i \oplus V_i)
\end{equation}
of matrices of the form 
\begin{equation}
    \begin{pmatrix}
        g_i && 0 \\ x_i && 1
    \end{pmatrix}
\end{equation}
for $g_i \in GL(U_i)$ and $x_i \in \text{Hom}(U_i, V_i)$, the latter regarded as an additive group on which $GL(U_i)$ acts. The group action is defined by 
\begin{equation} \label{eq:groupact}
\begin{split}
    (g, x) \cdot \begin{pmatrix} B_{1, i} && I_i \end{pmatrix} & = g_i^{-1} \begin{pmatrix} B_{1, i} && I_i \end{pmatrix} \begin{pmatrix} g_i && 0 \\ x_i && 1 \end{pmatrix} \\ 
    (g, x) \cdot \begin{pmatrix} \alpha_i && \jmath_i \end{pmatrix} & = \begin{pmatrix} \alpha_i && \jmath_i \end{pmatrix} \begin{pmatrix} g_i && 0 \\ x_i && 1 \end{pmatrix} \\
    (g, x) \cdot \begin{pmatrix} B_{2, e} && 0 \\ J_{3, e} && \xi_{2, e} \end{pmatrix} & = \begin{pmatrix} g_{t(e)} && 0 \\ x_{t(e)} && 1 \end{pmatrix}^{-1} \begin{pmatrix} B_{2, e} && 0 \\ J_{3, e} && \xi_{2, e} \end{pmatrix} \begin{pmatrix} g_{s(e)} && 0 \\ x_{s(e)} && 1 \end{pmatrix} \\ 
    (g, x) \cdot \begin{pmatrix} B_{3, e} && 0 \\ -J_{2, e} && \xi_{3, e} \end{pmatrix} & = \begin{pmatrix} g_{s(e)} && 0 \\ x_{s(e)} && 1 \end{pmatrix}^{-1} \begin{pmatrix} B_{3, e} && 0 \\ -J_{2, e} && \xi_{3, e} \end{pmatrix} \begin{pmatrix} g_{t(e)} && 0 \\ x_{t(e)} && 1 \end{pmatrix}.
\end{split}
\end{equation}
Note the group action depends on the point $p \in X$. $\mathscr{W}_p$ is obviously $GL(U)$-invariant; it is less obvious but easy to verify that it is in fact $GL(U) \ltimes \text{Hom}(U, V)$-invariant. Define the stable locus 
\begin{equation}
    \mathscr{R}(\vb{v}, \vb{w}; d)^{\text{st}} \hookrightarrow \mathscr{R}(\vb{v}, \vb{w}; d)
\end{equation}
of quiver representations satisfying the stability condition $\mathbb{C}[B_{1, i}] I_i(V_i) = U_i$ for each $i \in Q_0$. Note the stable locus is preserved by $GL(U) \ltimes \text{Hom}(U, V)$. $\mathscr{W}_p$ descends to a function on the stack quotient:
\begin{equation}
\begin{tikzcd}
    \mathscr{R}(\vb{v}, \vb{w}; d) \arrow[r, "\mathscr{W}_p"] & \mathbb{A}^1 \\
    \mathscr{R}(\vb{v}, \vb{w}; d)^{\text{st}} \arrow[u, hook] \arrow[r, two heads] & \left[ \mathscr{R}(\vb{v}, \vb{w}; d)^{\text{st}}/GL(U) \ltimes \text{Hom}(U, V) \right] \arrow[u, "\overline{\mathscr{W}}_p"] 
\end{tikzcd}
\end{equation}
and we define its critical locus as 
\begin{equation}
    M_X(p, d) := \text{Crit}(\overline{\mathscr{W}}_p)
\end{equation}
which may be viewed as a moduli stack of quiver representations. 

Note on $M_X(p, d)$ we have tautological vector bundles of rank $d_i$, which by abuse of notation we continue to denote by $U_i$ (similar to our notation for the universal quasimap bundles above). We also have the associated bundles
\begin{equation}
    F_i := \mathscr{R}(\vb{v}, \vb{w}; d)^{\text{st}} \times_{GL(U) \ltimes \text{Hom}(U, V)} (U_i \oplus V_i)
\end{equation}
which restrict to vector bundles on $\text{Crit}(\overline{\mathscr{W}}_p)$. These fit, for each $i \in Q_0$, into an exact sequence of vector bundles
\begin{equation}
\begin{tikzcd}
    0 \arrow[r] & V_i \otimes \mathscr{O}_{M_X(p, d)} \arrow[r] & F_i \arrow[r] & U_i \arrow[r] & 0
\end{tikzcd}
\end{equation}
over $M_X(p, d)$.

The main result of the paper is 

\begin{theorem} \label{thm:quasimapcrit}
    Notations as above. For any Nakajima variety $X$ and any choice of point $p \in X$, there is an isomorphism of algebraic stacks 
    \begin{equation}
        \mathsf{QM}^p_d(X) \simeq M_X(p, d). 
    \end{equation}
\end{theorem}

\subsection{Beilinson resolution} \label{beilinson}
The basic principle behind the proof of Theorem \ref{thm:quasimapcrit} is the following observation. $B$ is an arbitrary base scheme; $p: B \times \mathbb{P}^1 \to B$ and $q: B \times \mathbb{P}^1 \to \mathbb{P}^1$ are the projections. For any locally free sheaf $\mathscr{V}$ on $B \times \mathbb{P}^1$, denote $\mathscr{V}(k) := \mathscr{V} \otimes q^* \mathscr{O}_{\mathbb{P}^1}(k)$. Let $\mathscr{E}$ be a vector bundle on $B \times \mathbb{P}^1$ such that, for every point $b \hookrightarrow B$,
\begin{equation} \label{vanishassumpt}
    H^0 (\mathbb{P}^1, \mathscr{E}(-1) |_{p^{-1}(b)}) = 0
\end{equation}
and let $\mathscr{F}$ be another such bundle. Note $H^0(\mathbb{P}^1, \mathscr{E}(-2) |_{p^{-1}(b)}) = 0$ as well as a result of the inclusion of sheaves $\mathscr{E}(-2) |_{p^{-1}(b)}\hookrightarrow \mathscr{E}(-1) |_{p^{-1}(b)}$. Let $\phi: \mathscr{E} \to \mathscr{F}$ be a morphism of sheaves on $B \times \mathbb{P}^1$.  We consider $B \times \mathbb{P}^1 \times \mathbb{P}^1$, and let $\pi: B \times \mathbb{P}^1 \times \mathbb{P}^1 \to B \times \mathbb{P}^1$ be the projection onto the first and third factor, likewise $\pi
' : B \times \mathbb{P}^1 \times \mathbb{P}^1 \to B \times \mathbb{P}^1$ is the projection onto the first and second factor.

Let $w$ and $z$ be affine coordinates on the first and second $\mathbb{P}^1$. We regard the affine coordinate $z$ as a global section of the line bundle $\mathscr{O}_{\mathbb{P}^1}(1)$, equivalently as a meromorphic function on $\mathbb{P}^1$ with a simple pole at $\infty \in \mathbb{P}^1$. We have the resolution of the diagonal $\Delta \hookrightarrow \mathbb{P}^1 \times \mathbb{P}^1$:
\begin{equation}
\begin{tikzcd}
    0 \arrow[r] &  \mathscr{O}_{\mathbb{P}^1}(-1) \boxtimes \mathscr{O}_{\mathbb{P}^1}(-1) \arrow[r, "w - z"] & \mathscr{O}_{\mathbb{P}^1 \times \mathbb{P}^1} \arrow[r] & \mathscr{O}_\Delta \arrow[r] & 0. 
\end{tikzcd}
\end{equation}
Pulling this back under the flat morphism $B \times \mathbb{P}^1 \times \mathbb{P}^1 \to \mathbb{P}^1 \times \mathbb{P}^1$ and tensoring by the locally free sheaf $(\pi')^*(\mathscr{E}(-1))$, we have an induced commutative diagram 
\begin{equation}
\begin{tikzcd} 
0 \arrow[r] & \mathscr{E}(-2) \boxtimes \mathscr{O}_{\mathbb{P}^1}(-1) \arrow[r, "w - z"] \arrow[d, "\phi \boxtimes 1"] & \mathscr{E}(-1) \boxtimes \mathscr{O}_{\mathbb{P}^1} \arrow[d, "\phi \boxtimes 1"] \arrow[r] & (\pi')^*\mathscr{E}(-1) \otimes \mathscr{O}_{B \times \Delta} \arrow[d, "\phi"] \arrow[r] & 0  \\
0 \arrow[r] & \mathscr{F}(-2) \boxtimes \mathscr{O}_{\mathbb{P}^1}(-1) \arrow[r, "w - z"] & \mathscr{F}(-1) \boxtimes \mathscr{O}_{\mathbb{P}^1} \arrow[r] & (\pi')^*\mathscr{F}(-1) \otimes \mathscr{O}_{B \times \Delta} \arrow[r] & 0 
\end{tikzcd}
\end{equation}
exact on the top and bottom rows. We consider the associated long exact sequence of $R^i \pi_*$ and then tensor by $\mathscr{O}_B \boxtimes \mathscr{O}_{\mathbb{P}^1}(1)$, yielding 
\begin{equation} \label{eq:mainresolvediagram}
\begin{tikzcd}
    0 \arrow[r] & \mathscr{E} \arrow[r] \arrow[d, "\phi"] & R^1 p_*( \mathscr{E}(-2)) \boxtimes \mathscr{O}_{\mathbb{P}^1} \arrow[d] \arrow[r] & R^1 p_*(\mathscr{E}(-1)) \boxtimes \mathscr{O}_{\mathbb{P}^1}(1) \arrow[d] \arrow[r] & 0 \\
    0 \arrow[r] & \mathscr{F} \arrow[r] & R^1 p_*(\mathscr{F}(-2)) \boxtimes \mathscr{O}_{\mathbb{P}^1} \arrow[r] & R^1 p_*( \mathscr{F}(-1)) \boxtimes \mathscr{O}_{\mathbb{P}^1}(1) \arrow[r] & 0. 
\end{tikzcd}
\end{equation}
All vanishings follow from the hypothesis \eqref{vanishassumpt} above and the absence of higher cohomology on a point; we also used e.g. $$ R^1\pi_*( \mathscr{E}(-2) \boxtimes \mathscr{O}_{\mathbb{P}^1}(-1)) \simeq R^1 \pi_*( (\pi')^* (\mathscr{E}(-2)) \otimes q^* \mathscr{O}_{\mathbb{P}^1}(-1) \simeq R^1 p_*(\mathscr{E}(-2)) \boxtimes \mathscr{O}_{\mathbb{P}^1}(-1) $$ by the projection formula and flat base change. The unwritten vertical maps are induced by the functor $R^1 p_*( - \otimes q^* \mathscr{O}_{\mathbb{P}^1}(-k))$ applied to $\phi$, for $k = 1, 2$. A useful consequence of the commutativity of the left square and the exactness on the top and bottom rows is the following 
\begin{lemma} \label{mainlemma}
Let $\mathscr{E}, \mathscr{F}$ be vector bundles on $B \times \mathbb{P}^1$ satisfying the vanishing assumption \eqref{vanishassumpt}. Any morphism of sheaves $\phi: \mathscr{E} \to \mathscr{F}$ over $B \times \mathbb{P}^1$ is uniquely determined by the induced morphism $R^1 p_*(\mathscr{E}(-2)) \to R^1 p_*(\mathscr{F}(-2))$ of locally free sheaves on $B$. This is preserved under base change by arbitrary morphisms $B' \to B$. 
\end{lemma}
The adjective ``locally free'' and the part about base change follows from the vanishing assumption \eqref{vanishassumpt} and cohomology/base change (for a friendly introduction to the latter see e.g. Chapter 25 of \cite{vakil}). 

\section{Main construction}
With the preliminaries in place, we will prove Theorem \ref{thm:quasimapcrit} by constructing maps in both directions and showing they are inverse to each other. The proof occupies Sections \ref{quasi2quiver} and \ref{quiver2quasi}; in Section \ref{sect:splittings} we note a potential useful corollary. 

\subsubsection{Vanishing lemma}
The construction given below relies on the following crucial consequence of the fact that the quasimap $f = (\mathscr{V}, \mathscr{B}_2, \mathscr{B}_3, \mathscr{I}, \mathscr{J})$ evaluates at $\infty \in \mathbb{P}^1$ to a chosen point $p = (\xi_2, \xi_3, \iota, \jmath)$ in the stable locus. 

\begin{lemma} \label{lemma:canonicalsplit}
    For any point $ f = (\mathscr{V}, \mathscr{B}_2, \mathscr{B}_3, \mathscr{I}, \mathscr{J}) \hookrightarrow \mathsf{QM}^p(X)$ and each $i \in Q_0$, we have an exact sequence of sheaves on $\mathbb{P}^1$ as follows
    \begin{equation} \label{eq:splittingmap}
    \begin{tikzcd}
        0 \arrow[r] & \mathscr{V}_i \arrow[r] & V_i \otimes \mathscr{O}_{\mathbb{P}^1} \arrow[r] & \textnormal{torsion} \arrow[r] & 0.
    \end{tikzcd}
    \end{equation}
\end{lemma}

\begin{proof}
    Observe that the stability condition may be rephrased as the assertion that $V^*$ is generated under $\jmath$ by the action of $\xi_2, \xi_3$. Pick bases $e_{i \alpha}$, $\alpha = 1, \dots w_i$ in each framing space $W_i$, and write $\jmath_{i \alpha} \in V^*_i$ for the components of $\jmath$ in this basis (this is equivalent to the Crawley-Boevey trick). For any such pair $(m, \alpha)$ and any path $\gamma$ in $Q$ starting at $i \in Q_0$ and ending at $m \in Q_0$, define an element 
\begin{equation}
    j_{m \alpha} \xi_\gamma := j_{m \alpha} \overleftarrow{\prod}_{e \in \gamma} \, \, \, \xi_{\text{2 or 3}, e} \in V_i^*
\end{equation}
where $\xi_2$ is chosen if $\gamma$ agrees with the orientation of $Q$ and $\xi_3$ is chosen if $\gamma$ is opposite the orientation of $Q$. Let $\text{Path}_{im}(Q)$ denote the set of paths in $Q$ starting at $i$ and ending at $m$. That $(\xi_2, \xi_3, \iota, \jmath)$ is a stable quiver representation implies that, for each $i \in Q_0$, there is a subcollection of
\begin{equation}
    \Bigg\{ j_{m \alpha} \xi_\gamma \Bigg\}_{\substack{m \in Q_0 \\ 
    \alpha \in \{ 1, \dots, w_m \} \\ \gamma \in \text{Path}_{im}(Q)}} 
\end{equation}
which forms a basis in $V_i^*$; let $\sigma^{ia}$, $a = 1, \dots v_i$ denote such a basis, and $\tau_{ia}$ be the dual basis in $V_i$. Note for each quasimap $(\mathscr{V}, \mathscr{B}_2, \mathscr{B}_3, \mathscr{I}, \mathscr{J})$, $\sigma^{ia}$ induces a morphism of sheaves on $\mathbb{P}^1$, $\Sigma^{ia}: \mathscr{V}_i \to \mathscr{O}_{\mathbb{P}^1}$ defined by making the replacements $\jmath \mapsto \mathscr{J}$, $\xi_{2, 3} \mapsto \mathscr{B}_{2, 3}$. 

Then we may define a morphism of sheaves $\mathscr{V}_i \to V_i \otimes \mathscr{O}_{\mathbb{P}^1}$ by 
\begin{equation}
    \sum_{a = 1}^{v_i} \tau_{ia} \otimes \Sigma^{ia}. 
\end{equation}
It induces the identity map in the fiber over $\infty \in \mathbb{P}^1$ by construction, whence is a fiberwise isomorphism in a Zariski open neighborhood of $\infty$. This establishes that it is an injective sheaf homomorphism with torsion cokernel, which is what we wanted to prove. 
\end{proof}

To illustrate the proof of this lemma, consider the example where $X = \text{Hilb}_n(\mathbb{C}^2)$ and $p = \lambda$ is a monomial ideal labeled by partition $\lambda$. The boxes of $\lambda$ may be visualized with their corners at integer points $(i, j)$ on a two-dimensional coordinate grid. A distinguished basis in $V^*$ is given by $\{ \jmath \xi_2^{i - 1} \xi_3^{j - 1} \} _{(i, j) \in \lambda}$, and the inclusion $\mathscr{V} \hookrightarrow V \otimes \mathscr{O}_{\mathbb{P}^1}$ is given on local sections by $v(z) \mapsto (\mathscr{J}(z) \mathscr{B}_2^{i - 1}(z) \mathscr{B}_3^{j - 1}(z)v(z))_{(i, j) \in \lambda}$. 

\begin{cor} \label{cor:vanishforQM}
In any family of quasimaps over a base scheme $B$, for each point $b \hookrightarrow B$ we have 
\begin{equation}
    H^0(\mathbb{P}^1, \mathscr{V}_i(-1) |_{b \times \mathbb{P}^1}) = 0
\end{equation}
for each bundle $\mathscr{V}_i$, $i \in Q_0$. 
\end{cor}

\begin{proof}
    Pull the family back to $b$, tensor \eqref{eq:splittingmap} by $\mathscr{O}_{\mathbb{P}^1}(-1)$ and go to the long exact cohomology sequence.
\end{proof}

Fix a degree vector $d = (d_i)_{i \in Q_0}$. Let $p: \mathsf{QM}^p_d(X) \times \mathbb{P}^1 \to \mathsf{QM}^p_d(X)$, $q: \mathsf{QM}^p_d(X) \times \mathbb{P}^1 \to \mathbb{P}^1$ be the natural projections, and consider the universal bundles $\mathscr{V}_i$ over $\mathsf{QM}^p_d(X) \times \mathbb{P}^1$.
\begin{cor}
For each $i \in Q_0$, and each integer $k > 0$, $$R^1p_*(\mathscr{V}_i \otimes q^* \mathscr{O}_{\mathbb{P}^1}(-k))$$
is a vector bundle of rank $d_i + (k - 1) v_i$ on the stack $\mathsf{QM}^p_d(X)$. 
\end{cor}

\begin{proof}
By Corollary \ref{cor:vanishforQM} and cohomology/base change. To compute the rank, take the fiber over any point of $\mathsf{QM}^p_d(X)$ and apply Riemann-Roch on $\mathbb{P}^1$. 
\end{proof}

\subsection{From quasimaps to quiver} \label{quasi2quiver}
In this section we construct a morphism $\Psi: \mathsf{QM}^p_d(X) \to M_X(p, d)$, which will be our candidate isomorphism. Because $M_X(p, d)$ involves a lot of data, it is worth spelling out what it means to give a morphism into it, before describing $\Psi$. 
\begin{lemma} \label{lemma:definemap}
$Y$ is an algebraic stack. A morphism $Y \to M_X(p, d)$ is equivalent to the following data. 
\begin{enumerate}
    \item For $i \in Q_0$, a vector bundle $\mathscr{U}_i$ over $Y$ of rank $d_i$.
    \item For $i \in Q_0$, a vector bundle $\mathscr{F}_i$ over $Y$ of rank $d_i + v_i$, together with a short exact sequence 
    \begin{equation}
    \begin{tikzcd}
        0 \arrow[r] & V_i \otimes \mathscr{O}_Y \arrow[r] & \mathscr{F}_i \arrow[r] & \mathscr{U}_i \arrow[r] & 0.
    \end{tikzcd}
    \end{equation}
    \item Morphisms $\zeta_i : \mathscr{F}_i \to \mathscr{U}_i$ of bundles over $Y$ (note these are distinct from the surjections one line above).
    \item Morphisms $\mathbb{J}_i : \mathscr{F}_i \to W_i \otimes \mathscr{O}_Y$ such that 
    \begin{equation}
    \begin{tikzcd}
        V_i \otimes \mathscr{O}_Y \arrow[r, hook] \arrow[rd, "\jmath_i"] & \mathscr{F}_i \arrow[d, "\mathbb{J}_i"] \\ & W_i \otimes \mathscr{O}_Y
    \end{tikzcd}
    \end{equation}
    commutes. 
    \item For $e \in Q_1$, morphisms $\beta_{2, e}: \mathscr{U}_{s(e)} \to \mathscr{U}_{t(e)}$, $\beta_{3, e} : \mathscr{U}_{t(e)} \to \mathscr{U}_{s(e)}$, $\mathbb{B}_{2, e}: \mathscr{F}_{s(e)} \to \mathscr{F}_{t(e)}$, $\mathbb{B}_{3, e}: \mathscr{F}_{t(e)} \to \mathscr{F}_{s(e)}$ such that there is a commutative diagram 
    \begin{equation}
    \begin{tikzcd}
        0 \arrow[r] & V_{s(e)} \otimes \mathscr{O}_Y \arrow[r] \arrow[d, "\xi_{2, e}"] & \mathscr{F}_{s(e)} \arrow[r] \arrow[d, "\mathbb{B}_{2, e}"] & \mathscr{U}_{s(e)} \arrow[r] \arrow[d, "\beta_{2, e}"] & 0 \\
        0 \arrow[r] & V_{t(e)} \otimes \mathscr{O}_Y \arrow[r] & \mathscr{F}_{t(e)} \arrow[r] & \mathscr{U}_{t(e)} \arrow[r] & 0 
    \end{tikzcd}
    \end{equation}
    and a similar one with $2$ replaced by $3$ and $s(e), t(e)$ exchanged. 
\end{enumerate}
These data must satisfy the following conditions. 
\begin{enumerate}
    \item[(C1)] For $e \in Q_1$, the square 
    \begin{equation} \label{eq:critsqaure1}
    \begin{tikzcd}
        \mathscr{F}_{s(e)} \arrow[r, "\zeta_{s(e)}"] \arrow[d, "\mathbb{B}_{2, e}"] & \mathscr{U}_{s(e)} \arrow[d, "\beta_{2, e}"] \\ 
        \mathscr{F}_{t(e)} \arrow[r, "\zeta_{t(e)}"] & \mathscr{U}_{t(e)}
    \end{tikzcd}
    \end{equation}
    commutes, and a similar one with $2$ replaced by $3$ and $s(e), t(e)$ exchanged. 
    \item[(C2)] Define $\mathbb{I}_i: W_i \otimes \mathscr{O}_Y \to \mathscr{F}_i$ as the composition 
    $$ \begin{tikzcd}
        W_i \otimes \mathscr{O}_Y \arrow[r, "\iota_i"] & V_i \otimes \mathscr{O}_Y \arrow[r, hook] & \mathscr{F}_i.
    \end{tikzcd}$$
    Then the composition $\zeta_i \mathbb{I}_i = 0$ for $i \in Q_0$.
    \item[(C3)] For $i \in Q_0$, the morphism 
    \begin{equation}
        \sum_{e \in t^{-1}(i)} \mathbb{B}_{2, e} \mathbb{B}_{3, e} - \sum_{e \in s^{-1}(i)} \mathbb{B}_{3, e} \mathbb{B}_{2, e} + \mathbb{I}_i \mathbb{J}_i = 0. 
    \end{equation}
\end{enumerate}
Finally, each point of $Y$ must map to a quiver representation satisfying the stability condition $\mathbb{C}[B_{1, i}] I_i (V_i) = U_i$. 
\end{lemma}

\begin{proof}
    The data in 1-5 are just from the definition of $M_X(p, d)$ as a quotient stack and the group action \eqref{eq:groupact}, once we recognize
    \begin{equation}
    \begin{split}
        \zeta_i & = \text{pullback of $\begin{pmatrix} B_{1, i} && I_i \end{pmatrix}$} \\ 
        \mathbb{J}_i & = \text{pullback of $\begin{pmatrix} \alpha_i && \jmath_i\end{pmatrix}$} \\
        \beta_{2, 3, e} & = \text{pullback of $B_{2, 3, e}$} \\ 
        \mathbb{B}_{2,3, e} & = \text{pullback of $\begin{pmatrix} B_{2, 3, e} && 0 \\ \pm J_{3, 2, e} && \xi_{2, 3, e} \end{pmatrix}$}.
    \end{split}
    \end{equation}
    The part about the stability condition is likewise obvious. The only nontrivial content of this lemma is that conditions C1-3 are equivalent to the map factoring through the critical locus of $\mathscr{W}_p$. In view of the above, these conditions arise by pullback from the equations (from C1-2)
    \begin{equation}
    \begin{split}
        \begin{pmatrix} B_{1, t(e)} && I_{t(e)}\end{pmatrix} \begin{pmatrix} B_{2, e} && 0 \\ J_{3, e} && \xi_{2, e} \end{pmatrix}&  = B_{2, e} \begin{pmatrix} B_{1, s(e)} && I_{s(e)} \end{pmatrix} \\
        \begin{pmatrix} B_{1, s(e)} && I_{s(e)}\end{pmatrix} \begin{pmatrix} B_{3, e} && 0 \\ -J_{2, e} && \xi_{3, e} \end{pmatrix} & = B_{3, e} \begin{pmatrix} B_{1, t(e)} && I_{t(e)} \end{pmatrix} \\
        \begin{pmatrix} B_{1, i}  && I_i\end{pmatrix} \begin{pmatrix} 0 \\ \iota_i \end{pmatrix} & = 0
    \end{split}
    \end{equation}
    and (from C3)
    \begin{equation}
        \sum_{e \in t^{-1}(i)} \begin{pmatrix} B_{2, e} && 0 \\ J_{3, e} && \xi_{2, e} \end{pmatrix} \begin{pmatrix} B_{3, e} && 0 \\ -J_{2, e} && \xi_{3, e} \end{pmatrix} - \sum_{e \in s^{-1}(i)}  \begin{pmatrix} B_{3, e} && 0 \\ -J_{2, e} && \xi_{3, e} \end{pmatrix} \begin{pmatrix} B_{2, e} && 0 \\ J_{3, e} && \xi_{2, e} \end{pmatrix} + \begin{pmatrix} 0 && 0 \\ \iota_i \alpha_i && \iota_i \jmath_i \end{pmatrix}  = 0.
    \end{equation}
    By inspection of the components, these are equivalent to $d \mathscr{W}_p = 0$ for $\mathscr{W}_p$ given by \eqref{eq:definepotential}; note the data $p = (\xi_2, \xi_3, \iota, \jmath)$ are fixed so we do not differentiate with respect to them. 
\end{proof}

\subsubsection{}
To describe $\Psi$, we then must give the data $(\mathscr{U}, \mathscr{F}, \zeta, \mathbb{J}, \beta, \mathbb{B})$ and show it satisfies the constraints above. Recall $z$ denotes the affine coordinate on $\mathbb{P}^1$, regarded as a global section of $\mathscr{O}_{\mathbb{P}^1}(1)$. $(\mathscr{V}, \mathscr{B}_2, \mathscr{B}_3, \mathscr{I}, \mathscr{J})$ denote the universal quasimap over $\mathsf{QM}^p_d(X) \times \mathbb{P}^1$. 

\begin{prop}
    Notations as above. Setting 
    \begin{equation}
    \begin{split}
        \mathscr{U}_i & = R^1p_*(\mathscr{V}_i \otimes q^* \mathscr{O}_{\mathbb{P}^1}(-1)) \\
        \mathscr{F}_i & = R^1 p_* (\mathscr{V}_i \otimes q^* \mathscr{O}_{\mathbb{P}^1}(-2)) \\ 
        \zeta_i & = \text{$R^1 p_*(-)$ applied to $z: \mathscr{V}_i(-2) \to \mathscr{V}_i(-1)$} \\
        \mathbb{J}_i & = \text{$R^1 p_*(- \otimes q^* \mathscr{O}_{\mathbb{P}^1}(-2))$ applied to $\mathscr{J}_i: \mathscr{V}_i \to W_i \otimes \mathscr{O}_{\mathsf{QM} \times \mathbb{P}^1}$} \\ 
        \beta_{2, 3} & = \text{$R^1 p_*(- \otimes q^* \mathscr{O}_{\mathbb{P}^1}(-1))$ applied to $\mathscr{B}_{2,3} : \mathscr{V} \to \mathscr{V}$} \\ 
        \mathbb{B}_{2, 3} & = \text{$R^1p_*(- \otimes q^* \mathscr{O}_{\mathbb{P}^1}(-2))$ applied to $\mathscr{B}_{2, 3} : \mathscr{V} \to \mathscr{V}$}
    \end{split}
    \end{equation}
    defines a morphism $\Psi : \mathsf{QM}^p_d(X) \to M_X(p, d)$. 
\end{prop}

\begin{proof}
    We just need to check that the criteria in Lemma \ref{lemma:definemap} are satisfied. First is to verify that $\mathscr{F}_i$ and $\mathscr{U}_i$ fit into an exact sequence of the desired form. By Corollary \ref{cor:vanishforQM} and (the argument leading to) Lemma \ref{mainlemma}, we conclude there is an exact sequence 
    \begin{equation}
    \begin{tikzcd}
        0 \arrow[r] & \mathscr{V}_i \arrow[r] & \mathscr{F}_i \boxtimes \mathscr{O}_{\mathbb{P}^1} \arrow[r, "\zeta_i - z \pi _i"] & \mathscr{U}_i \boxtimes \mathscr{O}_{\mathbb{P}^1}(1) \arrow[r] & 0 
    \end{tikzcd}
    \end{equation}
    of locally free sheaves on $\mathsf{QM}^p_d(X) \times \mathbb{P}^1$. $\pi_i$ denotes $R^1p_*(-)$ applied to the tautological inclusion $\mathscr{V}_i(-2) \hookrightarrow \mathscr{V}_i(-1)$. Taking the fiber over $\mathsf{QM}^p_d(X) \times \infty$ and using the identification $\mathscr{V}_i |_{\mathsf{QM} \times \infty} \simeq V_i \otimes \mathscr{O}_{\mathsf{QM}}$ gives an exact sequence
    \begin{equation}
    \begin{tikzcd}
        0 \arrow[r] & V_i \otimes \mathscr{O}_{\mathsf{QM}^p_d(X)} \arrow[r] & \mathscr{F}_i \arrow[r, "-\pi_i"] & \mathscr{U}_i \arrow[r] & 0 
    \end{tikzcd}
    \end{equation}
    of vector bundles over $\mathsf{QM}^p_d(X)$. 

    Applying (the argument leading to) Lemma \ref{mainlemma} to the morphism $\mathscr{J}_i : \mathscr{V}_i \to W_i \otimes \mathscr{O}_{\mathsf{QM} \times \mathbb{P}^1}$ we find it fits into a diagram
    \begin{equation}
    \begin{tikzcd}
        \mathscr{V}_i \arrow[r, hook] \arrow{rd}[swap]{\mathscr{J}_i} & \mathscr{F}_i \boxtimes \mathscr{O}_{\mathbb{P}^1} \arrow[d, "\mathbb{J}_i"] \\ 
        & W_i \otimes \mathscr{O}_{\mathsf{QM} \times \mathbb{P}^1}
    \end{tikzcd}
    \end{equation}
    Taking the fiber over $\mathsf{QM}^p_d(X) \times \infty$ and using that $\mathscr{J}_i(\infty) = \jmath_i$ gives the commutative diagram required for $\mathbb{J}_i$. 

    Applying (the argument leading to) Lemma \ref{mainlemma} to the morphism $\mathscr{B}_{2, e} : \mathscr{V}_{s(e)} \to \mathscr{V}_{t(e)}$, we find a diagram 
    \begin{equation}
    \begin{tikzcd}
        0 \arrow[r] & \mathscr{V}_{s(e)} \arrow[r] \arrow[d, "\mathscr{B}_{2, e}"] & \mathscr{F}_{s(e)} \boxtimes \mathscr{O}_{\mathbb{P}^1} \arrow[r, "\zeta - z \pi"] \arrow[d, "\mathbb{B}_{2, e}"] & \mathscr{U}_{s(e)} \boxtimes \mathscr{O}_{\mathbb{P}^1}(1) \arrow[r] \arrow[d, "\beta_{2, e}"] & 0 \\ 
        0 \arrow[r] & \mathscr{V}_{t(e)} \arrow[r] & \mathscr{F}_{t(e)} \boxtimes \mathscr{O}_{\mathbb{P}^1} \arrow[r, "\zeta - z \pi"] & \mathscr{U}_{t(e)} \boxtimes \mathscr{O}_{\mathbb{P}^1}(1) \arrow[r] & 0
    \end{tikzcd}
    \end{equation}
    with exact rows and such that the squares commute. We suppressed the subscripts on $\zeta$ and $\pi$ to reduce clutter. Taking the fiber over $\mathsf{QM}^p_d(X) \times \infty$ gives the required commutative diagram for the $(\beta_2, \mathbb{B}_2)$ data. The argument for $\mathscr{B}_3$ is nearly identical. Taking the fiber at $z = 0$ gives the commuting squares required by condition C1 of Lemma \ref{lemma:definemap}. 

    Applying (the argument leading to) Lemma \ref{mainlemma} to the morphism $\mathscr{I}_i: W_i \otimes \mathscr{O}_{\mathsf{QM} \times \mathbb{P}^1} \to \mathscr{V}_i$, we find it fits into a diagram 
    \begin{equation}
    \begin{tikzcd}
        \mathscr{V}_i \arrow[r, hook] & \mathscr{F}_i \boxtimes \mathscr{O}_{\mathbb{P}^1} \\ 
        & W_i \otimes \mathscr{O}_{\mathsf{QM} \times \mathbb{P}^1} \arrow{u}[swap]{\mathbb{I}_i} \arrow[lu, "\mathscr{I}_i"]
    \end{tikzcd}
    \end{equation}
    with $\mathbb{I}_i$ defined as in Lemma \ref{lemma:definemap}. Observing that the vertical arrow factors through the kernel of $\zeta_i - z \pi _i$, we see it satisfies the constraint C2 of Lemma \ref{lemma:definemap}. 

    Finally, observe that since the quasimap data satisfy the moment map conditions $\comm{\mathscr{B}_2}{\mathscr{B}_3} + \mathscr{I} \mathscr{J} = 0$, an application of Lemma \ref{mainlemma} shows that condition C3 of Lemma \ref{lemma:definemap} is satisfied. 

    That the quiver data described satisfy the stability condition is a consequence of Lemma \ref{lemma:stability} below. 
\end{proof}

The following observation completes the proof, and will also be useful in the next section. 

\begin{lemma} \label{lemma:stability}
    Let $Y$ be an algebraic stack with data $(\mathscr{U}, \mathscr{F}, \zeta)$ listed in 1-3 of Lemma \ref{lemma:definemap}. Let $\pi_i : \mathscr{F}_i \twoheadrightarrow \mathscr{U}_i$ be the projection. Define a sheaf on $Y \times \mathbb{P}^1$ by 
    \begin{equation}
        \mathscr{V}_i = \ker(\zeta_i - z \pi_i). 
    \end{equation}
    Then $\mathscr{V}_i$ is a vector bundle if and only if the quiver data in $\zeta_i$ satisfy the stability condition pointwise on $Y$.
\end{lemma}
This is essentially shown in \cite{nakajima2011hansaw}, but we reproduce the argument for convenience of the reader. 

\begin{proof}
    $\mathscr{V}_i$ is a vector bundle if and only if $\zeta_i - z \pi_i$ is fiberwise surjective at all geometric points of $Y \times \mathbb{P}^1$. Restricting to the fiber at any point, this reduces to a question of linear algebra and we may write $\zeta_i = \begin{pmatrix} B_{1, i} && I_i \end{pmatrix}$ and $\pi_i = \begin{pmatrix} 1 && 0  \end{pmatrix}$ as matrices. 

    First we show the stability condition implies the surjectivity. Assume $B_{1, i}$ and $I_i$ satisfy $\mathbb{C}[B_{1, i}] I_i(V_i) = U_i$. The image of $\zeta_i - z \pi _i$ obviously contains the image of $I_i$ by the matrix formulas above. Moreover the image is $B_{1, i}$-invariant: for any vectors $\psi$ and $\chi$ we have 
    \begin{equation}
        B_{1, i}( (B_{1, i} - z) \psi + I_i \chi) = (B_{1, i} - z)(B_{1, i} \psi + I_i \chi) + z I_i \chi = (\zeta_i - z \pi_i) \begin{pmatrix} B_{1, i} \psi + I_i \chi  \\ z\chi \end{pmatrix}. 
    \end{equation}
    The stability condition then says that the image must be all of $U_i$, establishing the fiberwise surjectivity. 

    Going in the other direction, assume $\zeta_i - z \pi_i$ is surjective in the fiber at every point. Let $S_i \subseteq U_i$ be any subspace invariant under $B_{1, i}$ and containing the image of $I_i$, and let $\text{Ann}(S_i) \subseteq U_i^*$ be its annihilator. The annihilator is also $B_{1, i}$-invariant, so if it is nonzero then it contains a nonzero eigenvector for $B_{1, i}$, call it $\phi$. Since $\phi \in \text{Ann}(S_i)$, it must annihilate $I_i$, so it satisfies $\phi(B_{1, i} - \lambda) = 0$ for some $\lambda \in \mathbb{C}$ and $\phi I_i = 0$. Therefore $\phi$ annihilates the image of $\zeta_i - \lambda \pi_i$ for some $\lambda$. By assumption, this map is surjective for any $\lambda$ so $\phi = 0$, thus $\text{Ann}(S_i) = 0$ and $S_i = U_i$. 
\end{proof}

\subsection{From quiver to quasimaps} \label{quiver2quasi}
Now we show that $\Psi$ is our sought-after isomorphism by constructing its inverse $\Phi: M_X(p, d) \to \mathsf{QM}^p_d(X)$. This morphism is equivalent to an $M_X(p, d)$-family of quasimaps $(\mathscr{V}, \mathscr{B}_2, \mathscr{B}_3, \mathscr{I}, \mathscr{J})$. 

\subsubsection{}
Recall the universal bundles $U_i$, $F_i$ over $M_X(p, d)$ discussed in Section \ref{sect:quivermoduli}. By abuse of notation we write the universal morphims between them by the same symbols as the quiver data, e.g. $B_{2, e} : U_{s(e)} \to U_{t(e)}$ is a morphism of vector bundles over $M_X(p, d)$. 

For $i \in Q_0$, define a vector bundle over $M_X(p, d) \times \mathbb{P}^1$ by 
\begin{equation}
    \mathscr{V}_i := \ker \begin{pmatrix} B_{1, i} - z && I_i \end{pmatrix}.
\end{equation}
This is indeed a vector bundle by Lemma \ref{lemma:stability}. By construction, it fits into an exact sequence 
\begin{equation}
\begin{tikzcd}
    0 \arrow[r] & \mathscr{V}_i \arrow[r] & F_i \boxtimes \mathscr{O}_{\mathbb{P}^1} \arrow[r] & U_i \boxtimes \mathscr{O}_{\mathbb{P}^1}(1) \arrow[r] & 0. 
\end{tikzcd}
\end{equation}

\subsubsection{}
It remains to define the sections $(\mathscr{B}_2, \mathscr{B}_3, \mathscr{I}, \mathscr{J})$. Start with $\mathscr{B}_2$. As usual from \eqref{eq:groupact} we have a tautological morphism $\mathbb{B}_{2, e} : F_{s(e)} \to F_{t(e)}$, given by 
\begin{equation}
    \mathbb{B}_{2, e} = \begin{pmatrix} B_{2, e} && 0 \\ J_{3, e} && \xi_{2, e} \end{pmatrix}.
\end{equation}
The claim is that when pulled back to $M_X(p, d) \times \mathbb{P}^1$ and restricted to $\mathscr{V}_{s(e)} \hookrightarrow F_{s(e)} \boxtimes \mathscr{O}_{\mathbb{P}^1}$, it factors through $\mathscr{V}_{t(e)} \hookrightarrow F_{t(e)} \boxtimes \mathscr{O}_{\mathbb{P}^1}$. To see this it suffices to observe the following
\begin{equation}
    B_{2, e} \begin{pmatrix} B_{1, s(e)} - z && I_{s(e)} \end{pmatrix} = \begin{pmatrix} B_{1, t(e)} - z && I_{t(e)} \end{pmatrix} \begin{pmatrix} B_{2, e} && 0 \\ J_{3, e} && \xi_{2, e} \end{pmatrix}
\end{equation}
which follows from the equations 
\begin{equation}
\begin{split}
    \pdv{\mathscr{W}_p}{B_3} = \comm{B_1}{B_2} + I J_3 & = 0 \\
    \pdv{\mathscr{W}_p}{J_2} = B_2 I - I \xi_2 & = 0
\end{split}
\end{equation}
arising from $d \mathscr{W}_p = 0$. So we see it makes sense to define a morphism $\mathscr{B}_{2, e}: \mathscr{V}_{s(e)} \to \mathscr{V}_{t(e)}$ by 
\begin{equation}
\begin{tikzcd}
    \mathscr{V}_{s(e)} \arrow[d, "\mathscr{B}_{2, e}"] \arrow[r, hook] & F_{s(e)} \boxtimes \mathscr{O}_{\mathbb{P}^1} \arrow[d, "\mathbb{B}_{2, e}"] \\ 
    \mathscr{V}_{t(e)} \arrow[r, hook] & F_{t(e)} \boxtimes \mathscr{O}_{\mathbb{P}^1}. 
\end{tikzcd}
\end{equation}
In a similar fashion we see that the morphism $\mathscr{B}_{3, e}$ defined by 
\begin{equation}
\begin{tikzcd}
    \mathscr{V}_{t(e)} \arrow[d, "\mathscr{B}_{3, e}"] \arrow[r, hook] & F_{t(e)} \boxtimes \mathscr{O}_{\mathbb{P}^1} \arrow[d, "\mathbb{B}_{3, e}"] \\ 
    \mathscr{V}_{s(e)} \arrow[r, hook] & F_{s(e)} \boxtimes \mathscr{O}_{\mathbb{P}^1}
\end{tikzcd}
\end{equation}
is well-defined by the equations $\pdv{\mathscr{W}_p}{B_2} = \pdv{\mathscr{W}_p}{J_3} = 0$. 

\subsubsection{}
It remains to define the sections $\mathscr{I}$, $\mathscr{J}$. From \eqref{eq:groupact} we have tautological morphisms $\begin{pmatrix} \alpha_i && \jmath_i\end{pmatrix}: F_i \to W_i \otimes \mathscr{O}_{M_X(p, d)}$, so we may define morphisms $\mathscr{J}_i$ by the diagrams
\begin{equation}
 \begin{tikzcd}
        \mathscr{V}_i \arrow[r, hook] \arrow{rd}[swap]{\mathscr{J}_i} & F_i \boxtimes \mathscr{O}_{\mathbb{P}^1} \arrow[d, "(\alpha_i \, \, \, \jmath_i  )"] \\ 
        & W_i \otimes \mathscr{O}_{M_X \times \mathbb{P}^1}.
    \end{tikzcd}
\end{equation}
There are morphisms $\mathscr{I}_i$ 
\begin{equation}
\begin{tikzcd}
    \mathscr{V}_i \arrow[r, hook] & F_i \boxtimes \mathscr{O}_{\mathbb{P}^1} \\ & W_i \otimes \mathscr{O}_{M_X \times \mathbb{P}^1} \arrow{u}[swap]{\big( \begin{smallmatrix} 0 \\ \iota_i \end{smallmatrix} \big)} \arrow[lu, "\mathscr{I}_i"] 
\end{tikzcd}
\end{equation}
which are indeed well-defined by the equation 
\begin{equation}
    \begin{pmatrix} B_{1, i} - z&& I_i \end{pmatrix} \begin{pmatrix} 0 \\ \iota_i\end{pmatrix} = 0 
\end{equation}
which follows from 
\begin{equation}
    \pdv{\mathscr{W}_p}{\alpha} = I \iota = 0. 
\end{equation}

\subsubsection{}
So far we have given the bundles $\mathscr{V}$ and sections $(\mathscr{B}_2, \mathscr{B}_3, \mathscr{I}, \mathscr{J})$ over $M_X(p, d) \times \mathbb{P}^1$. Now observe that the sections satisfy the moment map equations $\comm{\mathscr{B}_2}{\mathscr{B}_3} + \mathscr{I} \mathscr{J} = 0$ as a result of the equations 
\begin{equation}
    \comm{\mathbb{B}_2}{\mathbb{B}_3} + \begin{pmatrix} 0 \\ \iota_i \end{pmatrix} \begin{pmatrix} \alpha_i && \jmath_i \end{pmatrix} = 0
\end{equation}
which in turn follow from the equations 
\begin{equation}
\begin{split}
    \pdv{\mathscr{W}_p}{B_1} = \comm{B_2}{B_3} & = 0 \\ 
    \pdv{\mathscr{W}_p}{I} = J_2 B_2 - \xi_2 J_2 + J_3 B_3 - \xi_3 J_3 + \iota \alpha & = 0. 
\end{split}
\end{equation}
Moreover, by construction of $\mathscr{V}_i$, we have a natural identification $\mathscr{V}_i |_{M_X \times \infty} \simeq V_i \otimes \mathscr{O}_{M_X}$ with respect to which $(\mathscr{B}_2(\infty), \mathscr{B}_3(\infty), \mathscr{I}(\infty), \mathscr{J}(\infty)) = (\xi_2, \xi_3, \iota, \jmath) = p$. 

\subsubsection{}
The discussion of the past four subsections may be summarized as 
\begin{prop}
Let $(\mathscr{V}, \mathscr{B}_2, \mathscr{B}_3, \mathscr{I}, \mathscr{J})$ denote the vector bundles and sections over $M_X(p, d) \times \mathbb{P}^1$ defined above. These describe a family of quasimaps of degree $d$ with $f(\infty) = p$ and therefore a morphism $\Phi: M_X(p, d) \to \mathsf{QM}^p_d(X)$. 
\end{prop}

It is a routine verification using Lemma \ref{mainlemma} to show that $\Phi \circ \Psi$ sends a family of quasimaps to an isomorphic family, and that the isomorphism between them is unique. A similar statement holds for $\Psi \circ \Phi$ applied to quiver representations. This concludes the proof of Theorem \ref{thm:quasimapcrit}. 

\subsection{Splitting the universal extension} \label{sect:splittings}
Theorem \ref{thm:quasimapcrit} gives an explicit description of the quasimap moduli via a quiver and potential. The only aspect of this description which is somewhat nonstandard is the need to quotient by the non-reductive group $GL(U) \ltimes \text{Hom}(U, V)$. It will be explained in this section that one can in fact get an equivalent description using a quotient by $GL(U)$, at the expense of making further choices depending on the evaluation point $p$ at $\infty$ of $\mathbb{P}^1$ in a non-canonical way. 

\begin{prop} \label{prop:gaugefixunipotent}
    Each choice of basis in $V$ constructed in the proof of Lemma \ref{lemma:canonicalsplit} defines a global section of the affine bundle
    \begin{equation}
    \begin{tikzcd}
        {[ (\textnormal{Crit}(\mathscr{W}_p) \cap \mathscr{R}(\vb{v}, \vb{w}; d)^{\text{st}})/GL(U) ]} \arrow[r] & M_X(p, d).
    \end{tikzcd}
    \end{equation}
\end{prop}

\begin{proof}
Producing such a section is equivalent to producing a splitting of the tautological exact sequences
\begin{equation}
\begin{tikzcd}
    0 \arrow[r] & V_i \otimes \mathscr{O}_{M_X(p, d)} \arrow[r] & F_i \arrow[r] & U_i \arrow[r] & 0
\end{tikzcd}
\end{equation}
over $M_X(p, d)$. Under the isomorphism $\Psi$ of Theorem \ref{thm:quasimapcrit}, this is pulls back to 
\begin{equation}
\begin{tikzcd}
    0 \arrow[r] & V_i \otimes \mathscr{O}_{\mathsf{QM}} \arrow[r] & R^1p_*(\mathscr{V}_i(-2)) \arrow[r] & R^1 p_*(\mathscr{V}_i(-1)) \arrow[r] & 0
\end{tikzcd}
\end{equation}
over $\mathsf{QM}^p_d(X)$. Splittings are the same as morphisms of sheaves $$ R^1 p_*( \mathscr{V}_i(-2)) \to V_i \otimes \mathscr{O}_{\mathsf{QM}}$$ restricting to the identity on $V_i \otimes \mathscr{O}_{\mathsf{QM}} \hookrightarrow R^1p_*(\mathscr{V}_i(-2))$. Noting that we have an isomorphism 
\begin{equation}
    V_i \otimes \mathscr{O}_{\mathsf{QM}} \simeq R^1 p_* (V_i \otimes \mathscr{O}_{\mathsf{QM} \times \mathbb{P}^1}(-2)) 
\end{equation}
we may apply Lemma \ref{mainlemma} to see that splittings are in natural bijection with morphisms $\mathscr{V}_i \to V_i \otimes \mathscr{O}_{\mathsf{QM} \times \mathbb{P}^1}$ extending the given identification $\mathscr{V}_i \eval_{\mathsf{QM} \times \infty} \simeq V_i \otimes \mathscr{O}_{\mathsf{QM}}$. Now observe that the proof of Lemma \ref{lemma:canonicalsplit} works in the universal family and produces such a morphism.
\end{proof}
In practice, this says that one can trade the instruction to quotient by $\text{Hom}(U, V)$ by instead setting some of the components of the $\alpha, J_2, J_3$ variables to zero; which components depends on $p$ and the arbitrary basis choice made in constructing the splitting. Verifying this is a straightforward exercise in using the definition of $\Psi$. Notice that splittings constructed in this fashion do not exist globally over $X$; the obstruction to their existence is the nontriviality of the tautological bundles on $X$. By contrast, the non-reductive quotient descriptions fit easily into a family over $X$.

\subsubsection{Example: partial flags revisited}
To illustrate the above procedure, let us explain how to recover the description of quasimaps to $X = T^*(GL_n/P)$ (with evaluation point chosen in the zero section) via handsaw quiver varieties as in \cite{nakajima2011hansaw} by specializing the construction of this note to the partial flag quivers. The idea will be that we can gauge fix the $\text{Hom}(U, V)$ part of the quotient in $M_X(p, d)$ to eliminate some redundant arrows, and then partially solve the equations in $\text{Crit}(\mathscr{W}_p)$ to eliminate some remaining arrows, and what remains will be exactly the data of \cite{nakajima2011hansaw}. 

We consider $GL_n/P$ to parameterize $\ell$-step partial flags $V_1 \subset V_2 \subset \dots V_\ell \subset \mathbb{C}^n$, with $\dim V_i = n_i$, $n_1 < n_2 < \dots < n_\ell < n$. We have vertex set $Q_0 = \{ 1, \dots, \ell \}$, edge set $Q_1 = \{1, \dots, \ell - 1\}$ and $s(i) = i$, $t(i) = i + 1$. 

We choose our fixed quiver representation for evaluation at $\infty$ as follows: we fix $V_i$ with $\dim V_i = n_i$, set $W_i = 0$ for $i \neq \ell$, and set $W_\ell = \mathbb{C}^n$. We set $\iota_\ell = 0$, $\xi_{3, i} = 0$, and assume that $\jmath_\ell: V_\ell \to W_\ell = \mathbb{C}^n$, $\xi_{2, i} : V_i \to V_{i + 1}$, $i = 1, \dots, \ell - 1$ are injective. We may as well choose bases in the $V_i \simeq \mathbb{C}^{n_i}$ so that each of $\xi_{2, i}$, $\jmath_\ell$ is the inclusion of a coordinate subspace. A basis choice from Lemma \ref{lemma:canonicalsplit} could be e.g. a maximal linearly independent subset of the components of $\{ \jmath_\ell, \jmath_\ell \xi_{2, \ell - 1}, \jmath_\ell \xi_{2, \ell - 1} \xi_{2, \ell - 2}, \dots, \jmath_{\ell} \xi_{2, \ell - 1} \dots \xi_{2, 1} \}.$

Fixing a degree vector $d = (d_1, \dots, d_\ell)$ and following the recipe of Theorem \ref{thm:quasimapcrit}, we introduce quiver data $I_i: V_i \to U_i$, $B_{1, i}: U_i \to  U_i$ ($i = 1, \dots, \ell$), $\alpha_\ell : U_\ell \to \mathbb{C}^n$, $B_{2, i}: U_i \to U_{i + 1}$, $J_{3, i}: U_i \to V_{i + 1}$, $B_{3, i}: U_{i + 1} \to U_i$, $J_{2, i}: U_{i + 1} \to V_i$ ($i = 1, \dots, \ell - 1$). The potential function is 
\begin{equation*}
\begin{split}
    \mathscr{W}_p & = \sum_{i = 1}^\ell \tr B_{1, i} B_{2, i - 1} B_{3, i - 1} - \sum_{i = 1}^{\ell - 1} \tr B_{1, i} B_{3, i} B_{2, i} \\
    & + \sum_{i = 1}^{\ell - 1} \tr J_{2, i}(B_{2, i} I_i - I_{i + 1} \xi_{2, i}) + \sum_{i = 1}^{\ell - 1} \tr J_{3, i}(B_{3, i} I_{i + 1}).
\end{split}
\end{equation*}
The critical point conditions are 
\begin{equation} \label{eq: flagcrit}
\begin{split}
   \pdv{\mathscr{W}_p}{B_{2, i}} = B_{3, i} B_{1, i + 1} - B_{1, i} B_{3, i} + I_i J_{2, i} & = 0 \\
   \pdv{\mathscr{W}_p}{J_{3, i}} = B_{3, i} I_{i + 1} & = 0 \\ 
   \pdv{\mathscr{W}_p}{B_{3, i}} = B_{1, i + 1} B_{2, i} - B_{2, i} B_{1, i} + I_{i + 1} J_{3, i} & = 0 \\ 
   \pdv{\mathscr{W}_p}{J_{2, i}} = B_{2, i} I_i - I_{i + 1} \xi_{2, i} & = 0 \\
   \pdv{\mathscr{W}_p}{B_{1, i}} = B_{2, i - 1} B_{3, i - 1} - B_{3, i} B_{2, i} & = 0 \\
   \pdv{\mathscr{W}_p}{I_i} = J_{2, i} B_{2, i} - \xi_{2, i - 1} J_{2, i - 1} + J_{3, i - 1} B_{3, i - 1} & = 0. 
\end{split}
\end{equation}
Notice that on the critical locus, because of the first four lines above, the stability condition $\mathbb{C}[B_1]I(V) = U$ we use in this paper is equivalent to imposing the stability condition $\mathbb{C} \langle B_1, B_2, B_3 \rangle I(V) = U$ (this is true in general). The second formulation will be more convenient to make contact with the handsaw. 

The action of $(x_1, \dots, x_\ell) \in \prod_{i = 1}^\ell \text{Hom}(U_i, V_i)$ on the quiver data $\alpha_\ell$, $J_{3, i}$ is from \eqref{eq:groupact}:
\begin{equation*}
\begin{split}
    \alpha_\ell & \mapsto \alpha_\ell + \jmath_\ell x_\ell \\ 
    J_{3, i} & \mapsto J_{3, i} + \xi_{2, i} x_i - x_{i + 1} B_{2, i}. 
\end{split}
\end{equation*}
Now we may use specific features of the flag variety to eliminate the quotient by $\prod_i \text{Hom}(U_i, V_i)$. Fixing a decomposition $$ \mathbb{C}^n \simeq \text{im} \, \jmath_\ell \oplus \mathbb{C}^{n - n_\ell} \simeq \mathbb{C}^{n_\ell} \oplus \mathbb{C}^{n - n_\ell}$$
we may write the components of $\alpha_\ell$ with respect to this decomposition as $\alpha_\ell = \begin{pmatrix} \alpha'_\ell && b_\ell\end{pmatrix}^T$, in terms of which the action is (because $\jmath_\ell$ is injective) $\alpha'_\ell \mapsto \alpha'_\ell + x_\ell$ and $b_\ell$ is unchanged. Clearly we may replace the quotient by $\text{Hom}(U_\ell, V_\ell)$ by the condition $\alpha'_\ell = 0$, and work with a new quiver involving only $b_\ell$. 

The remaining unipotent action on this new quiver is by the subgroup of the form $$(x_1, x_2, \dots, x_{\ell - 1}, 0) \in \prod_{i = 1}^{\ell - 1} \text{Hom}(U_i, V_i) \subset \prod_{i = 1}^\ell \text{Hom}(U_i, V_i).$$
and we have e.g. $J_{3, \ell - 1} \mapsto J_{3, \ell - 1} + \xi_{2, \ell - 1} x_{\ell - 1}$ under this subgroup. Similarly decomposing 
$$V_\ell  \simeq \text{im} \, \xi_{2, \ell - 1} \oplus \mathbb{C}^{n_\ell - n_{\ell - 1}} \simeq \mathbb{C}^{n_{\ell - 1}} \oplus \mathbb{C}^{n_\ell - n_{\ell - 1}}$$
we may write relative to this decomposition $J_{3, \ell - 1} = \begin{pmatrix} J'_{3, \ell - 1} && b_{\ell - 1} \end{pmatrix}^T$ and gauge fix by imposing $J'_{3, \ell - 1} = 0$. By induction, we may assume we have fixed decompositions 
$$V_i \simeq \text{im} \, \xi_{2, i - 1} \oplus \mathbb{C}^{n_i - n_{i - 1}} \simeq \mathbb{C}^{n_{ i- 1}} \oplus \mathbb{C}^{n_i - n_{i - 1}}$$
with respect to which $J_{3, i - 1} = \begin{pmatrix} 0 && b_{i - 1} \end{pmatrix}^T$. The remaining unipotent action is now trivial. Then we may replace the original quiver by a new quiver involving arrows $b_i: U_i \to \mathbb{C}^{n_{i + 1} - n_i}$ for $i = 1, \dots, \ell$ in place of $J_{3, i}$, $\alpha_\ell$ involving a quotient by $\prod_i GL(U_i)$ only. 

Such a quiver still has too many arrows to be the one from \cite{nakajima2011hansaw}, but at this point we may start using the equations \eqref{eq: flagcrit}. The $i = \ell$ component of the final line of \eqref{eq: flagcrit} reads 
$$ J_{3, \ell - 1} B_{3, \ell - 1} = \xi_{2, \ell - 1} J_{2, \ell - 1}$$
but by hypothesis, $J_{3, \ell - 1}$ factors through a complementary direct summand to the image of $\xi_{2, \ell - 1}$. Therefore $\xi_{2, \ell - 1} J_{2, \ell - 1} = 0$, which by injectivity of $\xi_2$ implies $J_{2, \ell - 1} = 0$. Continuing by induction, one infers that $\pdv{\mathscr{W}_p}{I_i} = 0$ implies $J_{2, i - 1} = 0$. In principle, we still have to deal with the remaining constraints $J_{3, i - 1} B_{3, i - 1} = 0$ on $J_3$, but these will turn out to be vacuous as we see presently. 

Using that $J_{2, i}$ must vanish for all $i = 1, \dots, \ell -1$, the first, second, and fifth lines of \eqref{eq: flagcrit} read 
\begin{equation*}
\begin{split}
    B_{3, i} B_{1, i + 1} - B_{1, i} B_{3, i} & = 0 \\ 
    B_{3, i} I_i & = 0 \\ 
    B_{2, i - 1} B_{3, i - 1} - B_{3, i} B_{2, i} & = 0
\end{split}
\end{equation*}
which together with the stability condition $\mathbb{C}\langle B_1, B_2, B_3 \rangle I(V) = U$ imply $B_{3, i} =0$ for all $i = 1, \dots, \ell - 1$. At this point, all but the third and fourth lines of \eqref{eq: flagcrit} assert $0 = 0$, so it remains only to deal with the third and fourth lines. The equation $B_{2, i} I_i = I_{i + 1} \xi_{2, i}$ says that, if we write the components of $I_i$ relative to the decompositions $V_i \simeq \text{im} \, \xi_{2, i - 1} \oplus \mathbb{C}^{n_i - n_{i - 1}}$ as 
\begin{equation*}
    I_i = \begin{pmatrix} \star && a_i \end{pmatrix}
\end{equation*}
then $\star$ is determined via $a_i$ and $B_{2, i}$ for lower values of $i$, so it is redundant information in the quiver and we may dispense with $I_i$ in favor of arrows $a_i: \mathbb{C}^{n_i - n_{i - 1}} \to U_i$. Concretely, relative to the full inductive splitting 
$$ V_i \simeq \mathbb{C}^{n_1} \oplus  \dots \oplus \mathbb{C}^{n_{i - 1} - n_{i - 2}} \oplus  \mathbb{C}^{n_i - n_{i - 1}}$$
we have 
$$I_i = \begin{pmatrix} B_{2, i - 1} \dots B_{2, 1} a_1 && \dots && B_{2, i - 1} a_{i - 1} && a_i \end{pmatrix}.$$
Notice in particular that 
$$ I_{i + 1} J_{3, i} = \begin{pmatrix} \star && a_{i +1} \end{pmatrix} \begin{pmatrix} 0 \\ b_i \end{pmatrix} = a_{i + 1} b_i.$$
We can summarize the above as the following: by gauge fixing the quotient by $\text{Hom}(U, V)$, we have been able to eliminate some components of $\alpha_\ell, J_{3, i}$ and use the equations \eqref{eq: flagcrit} to eliminate some other arrows. The ``surviving'' components of each arrow in the quiver obtained by specializing Theorem \ref{thm:quasimapcrit} are indicated as 
\begin{equation}
\begin{split}
    \alpha_\ell: U_\ell \to W_\ell \simeq \mathbb{C}^n  & \leadsto b_\ell: U_\ell \to \mathbb{C}^{n - n_\ell} \\ 
    J_{3, i}: U_i \to V_{i + 1} & \leadsto b_i: U_i \to \mathbb{C}^{n_{i + 1} - n_i} \\ 
    J_{2, i}: U_{i + 1} \to V_i & \leadsto 0 \\ 
    B_{3, i}: U_{i + 1} \to U_i & \leadsto 0 \\
    I_i: V_i \to U_i & \leadsto a_i: \mathbb{C}^{n_i - n_{i - 1}} \to U_i. 
\end{split}
\end{equation}
We keep $B_{1, i} : U_i \to U_i$ and $B_{2, i}: U_i \to U_{i + 1}$ as is. These arrows must satisfy the stability condition\footnote{It may appear at first sight confusing that the varieties entering Theorem \ref{thm:quasimapcrit} are defined with a stability condition involving $B_1$ only, but are equivalent to varieties defined with a stability condition involving $B_1, B_2$. The point is that the stability condition of Theorem \ref{thm:quasimapcrit} involves \textit{all components of $I_i$}, whereas the handsaw stability condition involves only $a_i$. By the discussion above, the missing components of $I_i$ are exactly what is obtained by acting with $B_2$ on $a_i$, so using the handsaw relation $\comm{B_1}{B_2} + ab = 0$ one can see the two formulations of the stability condition are equivalent.} $\mathbb{C}\langle B_1, B_2 \rangle a( \oplus_i \mathbb{C}^{n_i - n_{i - 1}}) = U$. Then the only nontrivial critical point condition from \eqref{eq: flagcrit} is the third line, which becomes 
\begin{equation}
    B_{1, i} B_{2, i - 1} - B_{2, i - 1} B_{1, i - 1} + a_i b_{i - 1} = 0.
\end{equation}
The arrows $(B_1, B_2, a, b)$ satisfying the above equation, stability condition, and taken modulo $\prod_i GL(U_i)$ precisely define Nakajima's handsaw variety \cite{nakajima2011hansaw} (unfortunately, what Nakajima calls $B_1$ we have called $B_2$ and vice versa). The argument above shows that the critical locus entering Theorem \ref{thm:quasimapcrit}, for $X = T^*(GL_n/P)$ with evaluation point chosen in the zero section, may be identified with a handsaw quiver variety. 

Proposition \ref{prop:gaugefixunipotent} guarantees that the first part of the procedure above may always be executed in general, though in general there is no reason to expect the rest of the critical point equations for $\mathscr{W}_p$ to simplify. 
\newpage

\printbibliography

@misc{qmtogit,
      title={Stable quasimaps to GIT quotients}, 
      author={I. Ciocan-Fontanine and B. Kim and D. Maulik},
      year={2011},
      eprint={1106.3724},
      archivePrefix={arXiv},
      primaryClass={math.AG},
      url={https://arxiv.org/abs/1106.3724}, 
}

@article{adhm,
    author = {M. Atiyah and V. Drinfeld and N. Hitchin and Y. Manin},
    title = {Construction of instantons},
    journal = {Physics Letters A},
    year = {1978}, 
    volume = {65}, 
    issue = {3}, 
    pages = {185-187},
}

@misc{okounkovpcmi,
      title={Lectures on K-theoretic computations in enumerative geometry}, 
      author={A. Okounkov},
      year={2017},
      eprint={1512.07363},
      archivePrefix={arXiv},
      primaryClass={math.AG},
      url={https://arxiv.org/abs/1512.07363}, 
}

@article{Pandharipande_2009,
   title={Curve counting via stable pairs in the derived category},
   volume={178},
   ISSN={1432-1297},
   url={http://dx.doi.org/10.1007/s00222-009-0203-9},
   DOI={10.1007/s00222-009-0203-9},
   number={2},
   journal={Inventiones mathematicae},
   publisher={Springer Science and Business Media LLC},
   author={Pandharipande, R. and Thomas, R. P.},
   year={2009},
   month=may, 
    pages={407–447} 
}

@misc{BT,
    author = {T. Botta and S. Tamagni}, 
    title = {Quasimap critical cohomology, Coulomb branches, and quantum groups}, 
    year = {2026},
}

@misc{ptvv,
      title={Shifted Symplectic Structures}, 
      author={T. Pantev and B. Toen and M. Vaquie and G. Vezzosi},
      year={2013},
      eprint={1111.3209},
      archivePrefix={arXiv},
      primaryClass={math.AG},
      url={https://arxiv.org/abs/1111.3209}, 
}

@misc{nakajima2011hansaw,
      title={Handsaw quiver varieties and finite W-algebras}, 
      author={H. Nakajima},
      year={2011},
      eprint={1107.5073},
      archivePrefix={arXiv},
      primaryClass={math.QA},
      url={https://arxiv.org/abs/1107.5073}, 
}

@article{nakajimaALE,
    author = {H. Nakajima},
    title = {Instantons on ALE spaces, quiver varieties, and Kac-Moody algebras},
    journal = {Duke Math J.},
    year = {1994},
    volume={76},
    issue={2},
    pages={365-416}
}

@article{nakajimaKM, 
author={H. Nakajima}, 
title={Quiver varieties and Kac-Moody algebras}, 
journal={Duke Math J.}, 
year={1998},
volume={91}, 
issue={3},
pages={515-560}

}

@url{vakil,
author={R. Vakil}, 
title={The Rising Sea: Foundations of Algebraic Geometry}, 
url={https://math.stanford.edu/~vakil/216blog/FOAGjul3123public.pdf}, 
year={2023}, 
}

\end{document}